\documentclass[12pt,a4paper,twoside]{amsart}
\usepackage{amsmath}
\usepackage{amssymb}
\usepackage{amsfonts}

\numberwithin{equation}{section}

\date{}
\def\SL{\text{\rm SL}}
\def\PSL{\text{\rm PSL}}
\def\vp{\varphi}

\def\Ker{\text{\rm Ker}}

\def\Ga{\Gamma}

\def\bbf{\mathbb{R}}
\def\bbz{\mathbb{Z}}
\def\bbn{\mathbb{N}}

\theoremstyle{plain}
\newtheorem{thm}{Theorem}[section]
\newtheorem*{thm*}{Theorem}
\newtheorem{lemma}[thm]{Lemma}
\newtheorem{prop}[thm]{Proposition}

\theoremstyle{definition} 
\newtheorem{definition}[thm]{Definition}
\newtheorem*{definition*}{Definition}

\newtheorem{cor}[thm]
{Corollary}

\newtheorem*{claim*}{Claim}
\newtheorem{remark}[thm]{Remark}
\def\bbz{\mathbb{Z}}

\def\bbf{\mathbb{F}}

\def\be{\begin{equation}}
\def\ee{\end{equation}}

\def\calo{\mathcal{O}}

\theoremstyle{remark}  
\begin{document}

\title{{Finiteness Properties \break and Profinite Completions}}

\author{Alexander Lubotzky}\thanks{This work is supported by ERC, NSF and ISF}
\maketitle
\centerline{\smaller{Einstein Institute of Mathematics}}
\centerline{\smaller{ Hebrew University of Jerusalem}}
\centerline{\smaller{Jerusalem 91904, Israel}}
\centerline{\smaller{alex.lubotzky@mail.huji.ac.il}}

\baselineskip 16pt

\begin{abstract}
In this note we show that various (geometric/homological) finiteness properties are not  profinite properties.  For example for every $1 \le k, \ell \le \bbn$, there exist two finitely generated residually finite groups $\Ga_1$ and $\Ga_2$ with isomorphic profinite completions, such that $\Ga_1$ is strictly of type $F_k$ and $\Ga_2$ of type $F_\ell$.
\end{abstract}

\section{Introduction}

Let $\Ga$ be a finitely generated group.  What properties of $\Ga$ can be deduced from its finite quotients?  The question makes sense only for residually finite groups.  Moreover, two finitely generated groups $\Ga_1$ and $\Ga_2$ have the same set of (isomorphism classes of) finite quotients if and only if they have isomorphic profinite completion $\hat \Ga_1\simeq \hat\Ga_2$ (cf. \cite{DFPR}).  Let us therefore define:
\begin{definition} A property of groups ${\boldsymbol P}$ is called a {\bf profinite property} if whenever $\Ga_1 $ and $\Ga_2$ are finitely generated residually finite groups with $\hat \Ga_1 \simeq \hat \Ga_2$ and $\Ga_1$ has property ${\boldsymbol P}$, so does $\Ga_2$.

\end{definition}

In recent years there has been a growing interest in understanding what properties are profinite properties (cf.  [B1], \cite{La}, \cite{Ak} and the references therein and especially \cite{GZ} for a historical review and a systematic program of research). This resembles the quite intensive area of study of ``geometric properties", i.e. properties shared by all pairs of finitely generated groups with quasi-isometric Cayley graphs.

The current note was sparked by a lecture given by Martin Bridson in Park-City in July 2012, where he presented an example of two finitely generated residually finite groups $\Ga_1$ and $\Ga_2$ with isomorphic profinite completions, such that $\Ga_2$ is finitely presented while $\Ga_1$ is not.  So, in the above terminology, the property of being finitely presented is not a profinite property. See \cite{BR} for this and more.

Our main result is:

\begin{thm} For every $r$ and $s$ there exists finitely generated residually finite groups $\Ga_r$ and $\Ga_s$ such that $\Ga_r$ has property $F_r$ (and not $F_{r+1}$), $\Ga_s$ has $F_s$ (and not $F_{s+1})$ and $\hat \Ga_r$ is isomorphic to $\hat \Ga_s$.
\end{thm}

Recall that a countable group $\Gamma$ is said to have property $F_m$ if there exists an Eilenberg-MacLane complex $K(\Ga, 1)$ with finite $m$-skeleton. Every group is of type $F_0$.  Property $F_1$ is equivalent to being finitely generated while property $F_2$ is equivalent to being finitely presented.  Our theorem is therefore a far-reaching generalization of Bridson-Reid's example.

Let us denote $\phi (\Ga) = sup\{ m|\Ga$ has property $F_m\}$ and call $\phi(\Ga)$-the {\bf finiteness length} of $\Ga$.  So our theorem gives:

\begin{cor} The finiteness length is not a profinite property.
\end{cor}

The theorem is  deduced in \S 2 from various deep results on arithmetic groups over positive characteristic  function fields.  A similar trick is used to deduce the following somewhat surprising result:

\begin{prop} Being residually solvable (resp. residually nilpotent, residually $-p$) is not a profinite property.
\end{prop}

The same trick when applied in \S 3 for arithmetic groups in characteristic zero gives us

\begin{prop} {\rm (a)} Being torsion-free is not a profinite property.

{\rm (b)} Having trivial center is not a profinite property.
\end{prop}

Additional results on lattices in Lie groups give:

\begin{prop} Cohomological dimension is not a profinite property
\end{prop}

We conclude in \S 4 with some related remarks, questions and suggestions for further research.\

\medskip
\noindent{\bf Acknowledgments:} The author acknowledges useful conversations with Martin Bridson and Kevin Wortman during the above mentioned Park City conference.

\section{Arithmetic groups of positive characteristic}

We now prove a much stronger version of Theorem 1.2.

\begin{thm} For every $1 \le n \in \bbn$, there exist residually finite groups $\Ga_0, \Ga_1, \ldots, \Ga_n$ with isomorphic profinite completions and with $\phi(\Ga_i) = i $ for $0 \le i \le n$.

\end{thm}

The proof relies on some properties  of positive characteristic arithmetic groups.  Some remarkable results have been proven recently on the finiteness properties of these groups (cf. [BW1] [BW2]) but for our case the more classical results on $\SL_2$ suffice. So, let us formulate them in a way ready for us to use:

\begin{thm}[Stuhler \cite{St}] Fix a prime power $q$, and $\calo = \bbf_q[t]$.  Let $S$ be a set of irreducible polynomials in $\bbf_q [t]$ and
\begin{equation*}
\calo_S = \left\{ \frac{f(x)}{g(x)} \in \bbf_q (t) \left|\begin{matrix} &f(x), g(x) \in \bbf_q[t] \text{\ and\  }\\
&\text{the\ only\ irreducible\  }\\
&\text{divisors\ of\ } g(x) \text{are\ from\ } S
\end{matrix}\right\}\right. .
\end{equation*}

Then $\phi (\SL_2(\calo_S)) = |S|.$

\end{thm}

So, for $S = \emptyset, \SL_2(\calo_S) = \SL_2(\bbf_q [t])$ is of type $F_0$ but not $F_1$, i.e. not finitely generated, while if $|S| = 1$, $\SL_2(\calo_S)$ is finitely generated $(F_1)$ but not finitely presented.  On the other hand, if $|S|\ge 2$, $\SL_2(\calo_S)$ is always finitely presented.

\begin{thm}[Serre \cite{Se}] If $|S| \ge 1$, then $\SL_2(\calo_S)$ has the congruence subgroup property, i.e.
\begin{equation*} \widehat{\SL_2 (\calo_S)} = \SL_2 (\hat\calo_S).\end{equation*}
\end{thm}
Let us spell out the concrete meaning of the last result: $\hat\calo_S$-the profinite completion of $\calo_S$ is equal to: $\mathop{\Pi}\limits_{p\notin S} (\calo_S)_{\hat p}$ where $(\calo_S)_{\hat p}$ is the completion of $\calo_S$ with respect to the topology of $\calo_S$ determined by the ideal $(p)$ generated by the irreducible polynomial $p$  (and its powers).  It is easy to see that $(\calo_S)_{\hat p} \simeq \calo_{\hat p}$.

\begin{remark} If $S$ is an infinite set of irreducible polynomials, and $\calo_S$ is defined in the same way as for a finite set, then Theorem 2.3 is still valid for $\SL_2(\calo_S)$.  This follows either from the proof in [Se] or from the fact that such $\SL_2 (\calo_S)$ is the union of $\SL_2(\calo_{S'})$ where $S'$ runs over the finite subsets of $S$. Note also that as long as $S$ is not the set of {\bf all} irreducible polynomials in $\bbf_q[x]$, $\SL_2(\calo_S)$ is residually finite since $\SL_2 (\calo_S) \hookrightarrow \SL_2 ((\calo_S)_{\hat p} )$, for every $p \notin S$,
and the latter group is a profinite group.  Of course, if $S$ is the set of all irreducible polynomials then $\calo_S = \bbf_q (t)$ and $\SL_2 (\bbf_q(t))$ has no finite index subgroup (in fact, it is simple mod its center).
\end{remark}

Before moving on to the proof of the theorem, we need the following lemma which is surely well known to experts.  We were not able to allocate an explicit reference.  We are grateful to Shmuel Weinberger who showed us how to deduce it quickly from \cite{Wa}.

\begin{lemma} Let $\Ga_1 $ and $\Ga_2$ be two countable groups.  Then $\phi(\Ga_1\times \Ga_2) = \min(\phi(\Ga_1), \phi(\Ga_2))$.
\end{lemma}

\begin{proof} By a result of Wall [Wa, Theorem A, p.~58] it is equivalent for a homotopy type to have property $F_m$ or to be dominated by a space of type $F_m$.  Thus, if $\Ga_1\times \Ga_2$ has type $F_m$ so do $\Ga_1$ and $\Ga_2$, i.e. $\phi(\Ga_i) \ge \phi(\Ga_1\times \Ga_2)$.  On the other hand, if both $K(\Ga_1, 1)$ and $K(\Ga_2, 1)$ have finite $m$-skeleton, so does $K(\Ga_1 \times \Ga_2, 1) = K(\Ga_1, 1) \times K(\Ga_2, 1)$ and hence $\phi(\Ga_1\times \Ga_2) \ge \min \{ \phi (\Ga_i) | i = 1, 2\}$ and the Lemma is proven.
\end{proof}

We are now ready to prove Theorem 2.1:  Fix $1\le n \in \bbn$ and fix a set of $n$  irreducible polynomials $S = \{ p_1, \ldots, p_n\}$ in $\bbf_q[x]$.  Denote for $i = 1,\ldots, n$, $S_i = \{ p_1,\ldots, p_i\}$.  Now choose for some $m\ge n$, a set $T$ of $m$ irreducible polynomials  with $T\cap S = \emptyset$.  For $i = 1,\ldots, n$, denote $T_i = T \cup \{ p_{i + 1,\ldots,} p_n\}$, so $T_n = T$ and $S_i \cup T_i = S \cup T$.  Finally let $\Ga_i = \SL_2 (\calo_{S_i}) \times \SL_2 (\calo_{T_i})$.

We claim:
(a)\ \  $\phi (\Ga_i) = i$.  Indeed by Theorem 2.2, $\phi (\SL_2(\calo_{S_i})) = i$ while $\phi (\SL_2(\calo_{T_i})) = m + n - i$.
Hence, by Lemma 2.5,  $\phi (\Ga_i) = \min (i, m+n-i) = i$.

(b) \ $\hat\Ga_1\simeq \hat \Ga_2 \simeq \cdots \simeq \hat\Ga_n$.  In fact, as $S_i \cup T_i = S\cup T$, we have by Theorem 2.3 and the explanation following it:
\begin{align*} \hat\Ga_i =& \prod\limits_{p\notin S_i} \SL_2 (\calo_{\hat p}) \times \prod_{p\notin T_i} \SL_2 (\calo_{\hat p}) \\
\cong& \prod_{p\notin T\cup S} (\SL_2 (\calo_{\hat p}) \times \SL_2 (\calo_{\hat p}))
\times  \prod\limits_{p \in T\cup S} \SL_2(\calo_p).\end{align*}
This shows that the isomorphism type of the profinite completion is independent of $i$.

We still have to show that there exists a countable group $\Ga_0$ which is residually finite and not finitely generated (so $\phi (\Ga_0) = 0)$ and with the same profinite completion as of $\Ga_1, \ldots, \Ga_n$.  To this end let $R$ be an infinite set of irreducible polynomials in $\bbf_q [t]$, whose complement $\bar R$ is nonempty.
In this case $\SL_2 (\calo_R)$ is an ascending union of finitely generated groups and  hence not finitely generated. Still by Remark 2.4, Serre's Theorem applies and both $\SL_2 (\calo_R)$ and $\SL_2 (\calo_{\bar R})$ have the congruence subgroup property. Hence
\begin{equation*}
\widehat{\SL_2 (\calo_R) \times \SL_2 (\calo_{\bar R})} = \SL_2 (\hat\calo_R) \times \SL_2 (\hat\calo_{\bar R}) = \prod_p \SL_2 (\calo_{\hat p})\end{equation*}
where this time $p$ runs exactly once over all the irreducible polynomials in $\bbf_q [t]$.

Let us now take $\Ga_0 = \SL_2 (\calo_{T\cup S}) \times \SL_2 (\calo_R) \times \SL_2 (\calo_{\bar R})$.
This is not a finitely generated group and its profinite completion is
\begin{align*} &\prod_{p\notin T\cup S} \SL_2 (\calo_{\hat p}) \times \prod_{\text{ all\ } p} \SL_2 (\calo_{\hat p}) \\
= &\prod_{p\notin T\cup S} (\SL_2 (\calo_{\hat p}) \times \SL_2 (\calo_{\hat p}))
\times
\prod_{p\in T\cup S} \SL_2 (\calo_{\hat p}).
\end{align*}
So, $\hat \Ga_0$ is isomorphic to $\hat \Ga_i$ for all $i = 1,\ldots, n$ and Theorem 2.1 is proven.

\subsection*{Proof of Proposition 1.4}

Let us first say that by residually solvable (resp. nilpotent, $p$) we mean here that the homomorphisms to {\bf finite} solvable (resp.~nilpotent, $p$) groups separate the points of the group.  But, actually the proof will work also in the other sense, i.e. when no finite assumption is made.  Anyway Proposition 1.4 is somewhat surprising since it shows that there are  residually finite groups $\Ga_1$ and $\Ga_2$  with the same finite quotients, and in particular, the same finite solvable quotients; still, for $\Ga_1$, the finite solvable quotients separate its points, while for $\Ga_2$ they do not.

For the proof we will use again the notations used in the proof of Theorem 2.1.  For simplicity, assume $q \ge 4$.
 Let $S_1 = \{ p_1, p_2\}$ be a set of two primes in $\bbf_q [x]$.
 Write $\Ga_S $ for $\SL_2 (\calo_S)$ and for a prime $p$, $p\notin S, \; \Ga_S(p) = \Ker \big(\SL_2(\calo_S) \to \SL_2(\calo_S/(p))\big)$, the congruence subgroup $\mod p$.

 Let now $p_3, p_4$ be two different primes not in $S$. Denote
 \begin{align*}
 &\Ga_1 = \Ga_S(p_3) \times \Ga_S (p_4) \\
  \text{and\ } &\Ga_2 = \Ga_S \times \Ga_S (p_3p_4),\\
   \text{where\ }  &\Ga_S(p_3p_4) = \Ga_S(p_3) \cap \Ga_S (p_4)
 \end{align*}
is the congruence subgroup $\mod p_3p_4$.
Now, $\Ga_1$ is a finitely presented residually $-p$ group.  Indeed $\Ga_S (p_3)$ is a subgroup of its closure in $\SL_2(\calo_S)_{\hat p_3})$, i.e. $\SL_2$ over the  $p_3$-closure of $\calo_S$.
But it is inside $\Ker \big(\SL_2 (\calo_S)_{\hat p_3} \to \SL_2((\calo_S)_{\hat p_3}/(p_3))\big)$ - the $p_3$-congruence subgroup which is a pro-$p$ group (this time $p$ is the rational prime such that $q=p^e$ for some $e$).  Hence $\Ga_S(p_3)$ is residually-$p$ group.  The same holds for $\Ga_S(p_4)$.  On the other hand, $\Ga_S$ has no non-trivial solvable quotient.  Indeed, by a well known result of Margulis, every normal subgroup of $\Ga_S$ is either  finite or of finite index, \cite[Chap.~VIII, Theorm (2.6), p.~265]{Ma}.  Moreover, from the affirmative solution of the congruence subgroup problem we deduce that every non-trivial finite quotient of $\Ga_S$ is mapped onto $\PSL_2(q^a)$ for some $a\ge 1$.  As $q\ge 4$, these are non-abelian finite simple groups.  This implies that $\Ga_2$ is not residually solvable.

Finally, by a similar argument as in the proof of Theorem 2.1, we see that
\begin{align*}
\hat\Ga_1=& \hat\Ga_S (p_3) \times \hat \Ga_S(p_4) = \prod_{p \neq p_3} \SL_2 (\calo_{\hat p})
\times \Ker \big(\SL_2(\calo_{\hat p_3})\to \SL_2(\calo_{\hat p_3} / (p_3))\\
&\times  \prod_{p \neq p_4}\SL_2 (\calo_p) \times \Ker \big(\SL_2(
\calo_{\hat p_4})\to \SL_2 (\calo_{\hat p_4}/ (p_4)\big) \\
\cong& \prod_{p \neq p_3,p_4} \SL_2 (\calo_{\hat p})^2\times\SL_2 (\calo_{\hat p_3}) \times \SL_2 (\calo_{\hat p_4})\\
&\times  \Ker (\SL_2(\calo_{\hat p_3}) \to (\SL_2(\calo_{\hat p 3}/ (p_3))) \times \Ker
(\SL_2(\calo_{\hat p_4}) \to \SL_2 (\calo_{\hat p_4} / (p_4))\big).
\end{align*}
While
\begin{align*}
\hat\Ga_2 =& \widehat{\SL_2 (\calo_S)} \times \widehat{\Ga_S(p_3p_4)}\\
=& \prod_p\SL_2(\calo_p) \times \prod_{p \neq p_3, p_4} \SL_2 (
\calo_p) \\
&\times  \Ker(\SL_2 (\calo_{\hat p_3} \to \SL_2 (\
\calo_{\hat p_3} / (p_3))\\
&\times  \Ker (\SL_2 (\calo_{\hat p_4}) \to \SL_2 (\calo_{p_4} / (p_4)) \end{align*}
and therefore 
$$   \hat \Ga_2 \simeq \hat\Ga_1.$$

Proposition 1.4 is now proven since $\Ga_1$ is residually finite-$p$, while $\Ga_2$ is not even residually solvable.
\qed
\section{Arithmetic groups of zero charactistic}

Let us start with proving Proposition 1.5:

This time let $\Gamma_S = \SL_2 (\bbz_S)$ where $S$ is a finite set of rational primes and
\begin{equation*}
\bbz_S = \left\{ \frac ab\,  \Big| \,  a, b \in \bbz \text{\ and\ all\ prime\ divisors\ of\ } b \text{\ are\ in\ } S\right\}.
\end{equation*}
As $\Ga_S$ contains $\SL_2(\bbz)$, it has nontrivial  torsion. But for every $2 \neq \ell \in \bbz$ which is not in $S$, the congruence subgroup $$ \Ga_S(\ell) = \Ker (\SL_2(\bbz_S) \to \SL_2 (\bbz_S / \ell \bbz_S))$$ is torsion free.

By a result of Serre \cite{Se} $\SL_2 (\bbz_S)$ has the congruence subgroup property whenever $S \neq \emptyset$.  This means that $\widehat{\SL_2(\bbz_S)} \simeq \SL_2 (\hat\bbz_S) = \prod\limits_{p\notin S}\SL_2(\bbz_p)$ when $\bbz_p$ is the ring of $p$-adic integers.  Now, for $\Gamma_S (\ell)$ we have
(still assuming $S\neq\emptyset$):
\begin{equation*} \widehat{\Ga_S(\ell)} \cong \Big(\prod\limits_{\underset{p \nmid \ell}{p\notin S}} \SL_2(\bbz_p)\Big)\times \prod\limits_{\underset{p \mid \ell}{p\notin S}} \Ker (\SL_2(\bbz_p) \to \SL_2(\bbz_p/\ell\bbz_p)
\end{equation*}

Now let $$S = \{ 7\}, \, \ell = 15 = 3\cdot 5, \, \Ga_1 = \SL_2 (\bbz_S) \times \Ga_S (\ell) \text{\ and\ } \Ga_2 = \Ga_S (3) \times \Ga_S (5).$$
Clearly $\Ga_1$ has torsion while $\Ga_2 $ does not.  Moreover $\Ga_1$ has nontrivial center, while $\Ga_2$ is centerless.  Still
 \begin{align*}\hat\Ga_1 = \Big( \prod_{p\neq 7} &\SL_2 (\bbz_p) \Big)\times \Big( \prod_{p\neq 7, 3, 5} \SL_2 (\bbz_p)\Big)\\ &\times \Ker(\SL_2(\bbz_3) \to \SL_2 (\bbf_3))
\times \Ker (\SL_2 (\bbz_5) \to \SL_2 (\bbf_5))\end{align*}
while
\begin{multline*}\hat\Ga_2 = \Big( \prod_{p\neq 3,7} \SL_2 (\bbz_p) \Big)\times \Ker\big(\SL_2(\bbz_3) \to \SL_2 (\bbf_3)\big)\\
\times \Big( \prod_{p\neq  5,7} \SL_2 (\bbz_p)\Big)
\times \Ker (\SL_2 (\bbz_5) \to \SL_2 (\bbf_5)).\end{multline*}
So both groups $\hat \Ga_1$ and $\hat \Ga_2$ are isomorphic to:
\begin{multline*} \prod_{p\neq 3,5,7} \Big(\SL_2 (\bbz_p)\times \SL_2 (\bbz_p)\Big) 
\times \SL_2(\bbz_3) \times \SL_2(\bbz_5)\\
\times \Ker \big(\SL_2(\bbz_3) \to \SL_2 (\bbf_3)\big)
\times\big( \Ker \SL_2 (\bbz_5) \to \SL_2 (\bbf_5)\big).\end{multline*}

This proves the proposition.
\subsection*{Proof of Proposition 1.6}

The efforts to answer the Grothendieck problem (cf. \cite{Gr},\cite{PT},\cite{BL},\cite{Py},\cite{BG}) have led to a number of methods and results of the following kind:  There exist finitely generated residually finite groups $\Ga_1$ and $\Ga_2$ with an injective map $\vp: \Ga_1 \to \Ga_2$, such that the induced map $\hat \vp:\hat\Ga_1\to \hat\Ga_2$ is an isomorphism while $\Ga_1$ is not isomorphic to $\Ga_2$.

Let us recall the construction from \cite{BL}:  there $\Ga_2 = L \times L$ when $L$ is a cocompact torsion free lattice in $G = Sp(n, 1)$ or $G = F_4^{(-20)}$. In particular the cohomological dimension $cd (\Ga_2) = 2cd(L)$ and $cd(L) = \dim(G/K)$ when $K$ is a maximal compact subgroup of $G$.  On the other hand $\Ga_1$ is obtained as follows:  Let $\pi:L\to M$ be an infinite finitely presented quotient of $L$ such that $M$ has no finite index subgroup. (Such a quotient $M$ exists by [Ol] as $L$ is hyperbolic group.)  Let $\Ga_1$  be the fiber product over $\pi$, i.e. $\Ga_1 = \{ (a, b) \in L \times L \Big| \pi(a) = \pi(b)\}$.  Then $\Ga_1$ is of infinite index, so $cd(\Gamma_1)  \lneqq cd (\Ga_2) = 2 cd (L)$ containing the diagonal subgroup (and so $cd(\Gamma_1) \ge cd(L)$).  It is shown in \cite{BL}that $\hat\Ga_1 \simeq \hat\Ga_2$  and hence Proposition 1.6 follows.

Let us remark, that our result here is not as strong as Theorem 1.2. We do not know to give, for arbitrary $r$ and $s$ in $\bbn$, examples of $\Ga_1$ and $\Ga_2$ with $cd(\Gamma_1) = r$, $cd (\Ga_2) = s$ and $\hat\Ga_1\simeq \hat\Ga_2$. This is probably a difficult problem: recall that $cd(\Ga) = 1$ if and only if $\Ga$ is a free group.  It is a long-standing open problem (usually attributed to Remeslenikov, cf.\cite{GZ}) whether freeness is a profinite property.

\section{Remarks and problems}

We have discussed throughout the paper only countable groups and especially finitely generated groups (which is the most interesting case for our problem) but one can also say something about uncountable groups:

By the recent remarkable result of Nikolov and Segal \cite{NS}, every finite index subgroup of a finitely generated profinite group $G$ is open.  This means that $\hat G = G$.  Applying this for $G = \hat \Ga$ the profinite completion of a finitely generated discrete group $\Ga$, we deduce that $\hat \Ga \simeq \hat G$.  Hence for every finitely generated infinite, residually finite discrete group $\Ga$, we deduce that $\hat \Ga \simeq \hat G$.  Hence for every finitely generated discrete group $\Ga$ there exists an uncountable group $G$ with $\hat \Ga = \hat G$.
On the other hand we do not know if the following is true: For every finitely generated  residually finite group $\Ga$ there exists a countable {\bf non} finitely generated (but residually  finite) group $G$, with $\hat \Ga = \hat G$.

In light of Remeslenikov's problem mentioned in the previous section, this would be an interesting problem to know if such $G$ exists when $\Ga$  is a finitely generated free group.

The main interest of \cite{BR}  is with pairs of groups $\Ga_1$ and $\Ga_2$ which have the same ``nilpotent genus" (i.e. for every $m \in \bbn, \; \, \Ga_1/\Ga_1^{(m)} \simeq \Ga_2/\Gamma_2^{(m)} $ where $\Gamma_i^{(m)} $ is the $m$ term of the lower central series of $\Ga_i$.) This implies (though not equivalent) to having the same pro(finite) nilpotent completion.
 Of course, if the profinite completions are isomorphic so are the pronilpotent completions.  But the examples we presented in the proofs of the results of this paper are usually not residually nilpotent.  This can be fixed quite easily by switching each time to a suitable congruence subgroup as we did in the proof of Proposition 1.4 (and recalling that proper principal congruence subgroups are always residually nilpotent).  We leave the details to the reader. This is usually easy except in the proof of Proposition 1.5, where some care has to be taken: Here it is important to replace $\Ga_1$ and $\Ga_2$ by their mod 2 congruence subgroups, as we want that the new $\Ga_1$ still has torsion and center.  (This is the reason  we presented the proof there with $S=\{ 7 \}$ and not with $S = \{ 2 \}$.  Of course for proving the original Proposition 1.5, we could use $S = \{2\}$ or $S = \{p\}$ for any prime $p\neq 3, 5$.  It is only for the pronilpotent version that it is important not to use $S = \{2\}$ as for this case $\Ga_S$ has {\bf no} mod 2 congruence subgroup.)

There is however an interesting difference between our residually nilpotent groups and the ones presented in \cite{BR}. There, all examples are residually torsion free nilpotent.  But our examples are never such.  The examples $\Ga_i$ presented for Theorems 1.2 (or 2.1) and Propositions 1.4 and 1.5 have the property $(FAb)$ i.e., for every finite index subgroup $\triangle,\; \, \triangle/[\triangle, \triangle]$ is finite, since they have the congruence subgroup property.  The same holds for $\Ga_2$ of the proof of Proposition 1.6, since $\Ga_2$ has Kazhdan property $T$.  (It follows that $\Ga_1$ also has $(FAb)$, since $(FAb)$ is clearly a profinite property). For all our examples, when $\Ga_1$ and $\Ga_2$ have the same pronilpotent completion they also have the same nilpotent genus.

\

\end{document}